\documentstyle[a4,twoside]{article}

\def\@begintheorem#1#2{\it \theoremhook \trivlist \item[\hskip \labelsep{\bf #1\ #2}]}
\def\@opargbegintheorem#1#2#3{\it \theoremhook \trivlist
      \item[\hskip \labelsep{\bf #1\ #2\ (#3)}]}
\def\theoremhook{}

\newtheorem{theo}{Theorem}[section]

\newtheorem{lemma}[theo]{Lemma}

\newenvironment{proof}{{\it Proof.}}%
        {\ifvmode\else\unskip\fi ~\penalty10000 \hfill%
        $\Box$\vspace{0ex}}

\newcommand{\CompN}{{\mbox{$\bf C$}}}
\newcommand{\dlik}{\mathop{:=}\nolimits}
\newcommand{\Gr}{\mathop{Gr}}
\newcommand{\into}{\mbox{$\: \rightarrow \:$}}
\newcommand{\Prj}{\mathop{Pr}}
\newcommand{\st}{\mid}
\newcommand{\Sp}{\mathop{Sp}\nolimits}

\begin{document}

\title{\raisebox{9ex}[1ex][.75ex]{\makebox[0ex][l]{
\hspace{-10ex}\footnotesize $
\stackrel{\mbox{{\it preprint Mathematics} No. 3/1995, 1-10}}
{\mbox{Dep. of Mathematics, Univ. of Tr.heim}}$}}
Measurability of Intersections of Measurable Multifunctions}
\author{Gunnar Taraldsen}
\date{} 

\maketitle

\vspace{-7ex}
\begin{quotation}
We prove universal compact-measurability of the intersection
of a compact-measurable Souslin family of 
closed-valued multifunctions.
This generalizes previous results
on intersections of measurable multifunctions.
We introduce the unique maximal part of a multifunction
which is defined on the quotient given by an equivalence relation.
Measurability of this part of a multifunction
is proven in a special case.
We show how these results apply to the spectral theory of
measurable families of closed linear operators.
\footnote{Published:
  Taraldsen, G. Rend. Circ. Mat. Palermo (1996) 45: 459. https://doi.org/10.1007/BF02844516}
\end{quotation}

\addtolength{\baselineskip}{0.5\baselineskip}
\markboth{\hfill {\it Gunnar Taraldsen} \hfill}{
\hfill {\it Intersection of Multifunctions} \hfill}

\section{Introduction.}

If 
$\omega \mapsto F (\omega)$ and 
$\omega \mapsto G (\omega)$
are set-valued functions ($\equiv$ multifunctions),
then it is reasonable to consider properties of
the multifunctions 
$F \cap G$ and $F \cup G$.
$F$ is compact-measurable
if $\{\omega \st F (\omega) \cap K = \emptyset \}$
is measurable for all compacts $K$.
Borel-, closed-, and open-measurability are defined likewise.
It is trivial to verify measurability of 
finite or countable unions of measurable multifunctions.
Himmelberg {\sl et al} \cite[p.166,ex.2]{HPV} give an example
of two closed-valued open-measurable multifunctions into a Polish space
for which the intersection is not open-measurable.
On the other side they prove compact-measurability of
the intersection of a countable family of closed-valued 
compact-measurable multifunctions into a metric space.
A metric space has second countable compacts.
We prove  compact-measurability of
the intersection of a countable family of closed-valued 
compact-measurable multifunctions into a Hausdorff space with second countable
compacts.
The idea of our proof is different from the 
proof of Himmelberg  {\sl et al} and was motivated by the separation 
ideas in the work of Spakowski  {\sl et al} \cite{SPAKOWSKI}.
Spakowski  {\sl et al} prove open-measurability of
the intersection of a countable family of compact-valued 
open-measurable multifunctions into a countably $C$-separated 
Hausdorff space.
We generalize this to the case of  
a countably $C$-separated 
$T_1$-space,
with a simplification of their proof.
We remark that a space is $C$-separated and $T_1$ if and only if
every compact set is closed.

The counterexample of Himmelberg {\sl et al}  
disappears when the $\sigma$-algebra is completed.
Castaing  {\sl et al} \cite{CASTAING} prove
universal  Borel-measurability of
the intersection of a countable family of closed-valued 
universally open-measurable multifunctions into a Souslin space.
Our main result on intersections is
the Souslin intersection theorem which gives
universal compact-measurability of the intersection of 
a compact-measurable Souslin family of closed-valued multifunctions
into a  Hausdorff space with second countable
compacts.

In the final section we consider multifunctions of the form
\[
H (\omega) \dlik \bigcup_{b \in G (\omega)} F (\omega, b)
\]
\[
K (\omega) \dlik \bigcap_{b \in G (\omega)}  F (\omega, b).
\]
Universal open-measurability of $H$ is proven,
but the corresponding proof for $K$ fails.
Of particular interest is the case
\[
F_\sim (\omega) \dlik \bigcap_{b \sim \omega}  F (b),
\]
where $\sim$ is an equivalence relation.
$F_\sim$ is the unique maximal part of $F$ which is defined on the
quotient given by $\sim$.
We apply the Souslin intersection theorem to prove
measurability of $F_\sim$ in a special case.

To motivate the reader further we give here an example of 
one of the many possible applications of the theory of
measurable multifunctions,
and in particular the Souslin intersection theorem.
Consider the Schr{\"o}dinger operator
\[
H_{x} = - \Delta + W_0 + \sum_{i = 1}^\infty x_i W_i, 
\]
where $\Delta$ is the Laplace differential operator,
the $W_i$'s are functions acting by multiplication,  
and $x$ is a sequence of numbers.
Under suitable assumptions $H_{x}$ is realized as
a closed linear operator acting in the Hilbert space of
square integrable functions \cite{PASTUR}.
$H_x$ is closed iff the graph
$G (H_x) \dlik \{(f, H_x f) \st f \in D (H_x) \}$ is
a closed set.
If the multifunction $x \mapsto G (H_x)$ is open-measurable,
we say that
$\{H_x \}$ is a measurable family of closed operators.
In this case it turns out that the spectrum of the operator
gives a closed-valued universally open-measurable
multifunction $x \mapsto \Sigma (x) \dlik \sigma (H_x)$.
$\Sigma$ is typically even lower semicontinuous
in the special case of the above  Schr{\"o}dinger operator.
It is of particular importance to understand the dependence
of the spectrum on the behavior of the sequence $x$ at infinity.
One possibility is to consider two sequences $x,y$ to be equal at
infinity $x \sim y$ if there exists an $N$ such that
for all $n \ge N$, $x_n = y_n$.
The interesting part of the spectrum is then the tail part
$\Sigma_\sim$, which is 
the unique maximal part of the spectrum defined on the
quotient given by $\sim$.
The Souslin intersection theorem gives open-measurability
of the tail part of the spectrum.
Equipp the set of sequences with the product measure $\mu = \otimes \mu_i$,
where the $\mu_i$'s are probability measures.
We prove existence of a closed set $A$ such that
$\mu \{x \st \Sigma_\sim (x) = A \} = 1$.

\section{Preliminaries.}

Let ${\cal E}$ be a $\sigma$-algebra of sets in $\Omega$,
so ${\cal E}$ is a family of subsets of  $\Omega$ -
closed under complements and countable unions.
The empty set is countable, 
so $\emptyset \in {\cal E}$.
$\Omega$ equipped with ${\cal E}$
is a measurable space.  
Let $\mu : {\cal E} \into [0,\infty)$ be a finite measure,
which means that the  measure of the union of a disjoint countable
family of measurable sets equals the sum of the measure of each set.
In particular $\mu (\emptyset) = 0$, 
since an empty sum is zero.
A measure space is a set equipped with a measure.
${\cal E}_\mu$ is the family of sets $B$ such that there exists
$A,C \in {\cal E}$ with $A \subset B \subset C$ and $\mu (C \setminus A) = 0$. 
The completion ${\cal E}_\mu$ of  ${\cal E}$ is a $\sigma$-algebra,
and $\mu$ extends to  ${\cal E}_\mu$.
The universal completion is 
$\hat {\cal E} \dlik \bigcap_{\nu}{\cal E}_\nu$,
for all positive finite measures $\nu$.
The sets in  ${\cal E}$, $\hat {\cal E}$, and  ${\cal E}_\mu$ are
denoted measurable, universally measurable, and $\mu$-measurable respectively.
Observe  ${\cal E} \subset \hat {\cal E} \subset {\cal E}_\mu$,
and in particular $\hat {{\cal E}_\mu} = {\cal E}_\mu$.
The $\sigma$-algebra generated by a family of subsets in $\Omega$
is the intersection of all  $\sigma$-algebras containing the family.
The $\sigma$-algebra $\otimes_{\alpha \in \Lambda} {\cal E}_\alpha$  is
the  $\sigma$-algebra generated by
$\{ \{\omega \in \prod_\alpha \Omega_\alpha \st 
\omega (\beta) \in A_\beta \} \st 
\beta \in \Lambda,\; A_\beta \in {\cal E}_\beta \}$.
$\otimes_{\alpha} {\cal E}_\alpha$ is the product $\sigma$-algebra in
$\prod_\alpha \Omega_\alpha$.

Let $X$ be a topological space.
This means that there is a family $\tau$ of subsets of $X$ -
closed under arbitrary unions and finite intersections.
The empty union and intersection imply in particular 
$\emptyset,X \in \tau$.
The sets in  $\tau$ are called open and the complement of an open
set is called closed.
$X$ is Hausdorff ($T_2$) if for any given $x \neq y$
there exist open disjoint sets $U,V$ with $x \in U, y \in V$.
$X$ is a $T_1$-space if every finite set is a closed set.
A family $\{A_\alpha \}$ of sets covers a set $K$ if
$K \subset \bigcup_\alpha A_\alpha$.
A set $K$ is compact if every family 
$\{U_\iota \}$ of open
sets which covers $K$ contains a finite family 
$\{U_{\iota_1}, \ldots, U_{\iota_N} \}$
which covers $K$.
A finite set is compact,
and in a Hausdorff space a compact set is closed.
$\tau$ is $C$-separating \cite{SPAKOWSKI} if for any given compact set
$K$ and a disjoint closed set $F$ there exists an open set
containing $F$,
 but not intersecting $K$.
A space is $T_1$ and $C$-separated if and only if
every compact set is closed.
$X$ is separable if there exists
a countable set $D \subset X$ such that 
$D \bigcap U \neq \emptyset$ for all open sets $D \neq \emptyset$.
A topology $\tau$ is second countable if there exists a
countable $\tau_0 \subset \tau$ such that any
open set is a union of sets from $\tau_0$. 
$X$ is countably $C$-separated if there exists a countable
$C$-separating  family of open sets.
In this case every open set is a countable union of closed sets.
If $K$ is a subset of $X$, 
then $\tau_K \dlik \{K \bigcap U \st U \in \tau \}$ is the relative topology
in $K$.
$X$ has second countable compacts
if $\tau_K$ is second countable whenever $K$ is compact.
The family ${\cal B}$ of Borel sets is the  $\sigma$-algebra
generated by the topology.

A metric $d$ on a set $S$ is a function
$d: S^2 \into [0,\infty)$ such that for all $x,y,z$:
$d(x,y) = 0\; \Leftrightarrow\; x = y$, 
$d(x,y) = d(y,x)$, and 
$d(x,z) \le d(x,y) + d(y,z)$.
The set $\{y \st d(x,y) < r \}$ is the open ball with center 
at $x$ and radius $r$. 
The family of sets which are unions of open balls is a topology.
A sequence $x_1, x_2, \ldots$ is convergent if there exists an
$x$ such that $d(x,x_1), d(x,x_2), \ldots$ converges to $0$.
The sequence is Cauchy if $d(x_n,x_m)$ converges to $0$ as
$n,m \rightarrow \infty$.
A metric space is complete if every Cauchy sequence is convergent.
A Polish space is a separable, complete metric space.
A Souslin space $S$ is a metric space which is the range
of a continuous function with
a Polish domain \cite{BOURBAKI.III}.
Let $A \in {\cal E} \otimes {\cal B}$, where 
${\cal E}$ is a $\sigma$-algebra in $\Omega$ and  ${\cal B}$
is the family of Borel sets in a Souslin space $S$.
The projection
\[
\Pr_\Omega A \dlik \{\omega \st (\omega,x) \in A \}
\]
need not be measurable,
but it is universally measurable by 
the Souslin projection theorem \cite[p.75, Thm. III.23]{CASTAING}.

A multifunction $F: \Omega \into X$ is a function
$F: \Omega \into {\cal P} (X)$.
This means that for each $\omega \in \Omega$ there is assigned
a subset $F (\omega)$ of $X$.
Let $U$ be a subset of $X$.
The set of $\omega$ such that $F (\omega)$ intersects $U$ is denoted by
\[
F^{-1} (U) \dlik \{\omega \st F (\omega) \cap U \neq \emptyset \} .
\]
In particular $F^{-1}$ is a multifunction 
$F^{-1} : {\cal P} (X) \into \Omega$.
If $f: \Omega \into X$ is a function,
then $\{f \} : \Omega \into X$ is the singelton-valued multifunction
$\omega \mapsto \{f (\omega) \}$.
We observe $f^{-1} = \{f\}^{-1}$,
so statements about a function $f$ which only involves $f^{-1}$
may have natural generalizations to multifunctions.
$F^{-1}$ preserves unions
\[
F^{-1} (\bigcup_\lambda A_\lambda) = \bigcup_\lambda F^{-1} (A_\lambda),
\]
but not intersections, since
\[
F^{-1} (A \bigcap B) \subset F^{-1} (A) \bigcap F^{-1} (B),
\]
and the opposite inclusion is false.
${f}^{-1}$ preserves unions,
the empty set, 
the full set,
intersections, 
complements,
and disjoint unions.
Only the first two of these six statements generalise
to $F^{-1}$.
The graph of $F$ is 
$\Gr F \dlik \{(\omega,x) \st x \in F(\omega) \}$.
$\Gr$ is
a bijection between
multifunctions from $\Omega$ to $X$
and subsets of $\Omega \times X$.
$\Gr$ preserves unions,
the empty set, 
the full set,
intersections, 
complements,
and disjoint unions.

Let $\Omega$ be a measurable space, 
let $X$ be a topological space,
and let $F: \Omega \into X$ be a multifunction.
$F$ is (universally) open-measurable if $F^{-1} (U)$ is 
(universally) measurable for all open sets $U$.
Borel-, closed-, compact-measurable is defined likewise.
If $F$ is singelton-valued,
then Borel-, closed-, and open-measurability
are equivalent. 
If $X$ is a metric space and $F$ is closed-valued, then
Borel- implies closed- implies open- implies compact-measurability,
and furthermore
\[
\Gr F = \{(\omega,x) \st d(x,F (\omega)) = 0 \},
\]
so $F$ has a measurable graph if 
$F$ is open-measurable.
If $X$ is Hausdorff and a countable union of compacts,
then compact- implies closed-measurability.
If $F$ is open-measurable and non-empty closed-valued into a
Polish space, then
\[
F (\omega) = \{x_1 (\omega), x_2 (\omega), \ldots \}^-
\]
for suitable measurable $x_i$'s.
This is the selection theorem \cite[p.65, Theorem III.6]{CASTAING}.
If $F$ has a measurable graph,
then $F$ is universally Borel-measurable.
This follows from the identity
\[
F^{-1} (A) = \Pr_\Omega [\Omega \times A \bigcap \Gr F],
\]
and the Souslin projection theorem.
A family $\{F_s \}_{s \in S}$
of multifunctions is compact-measurable
if the multifunction
$(\omega, s) \mapsto F_{s} (\omega)$ is compact-measurable.

\section{Intersection of Multifunctions.}

In this section we will prove the 
compact-measurability
results on intersection
of compact-measurable multifunctions.
Some examples and consequences are discussed in the final section.

Since the idea of countable separation
as found in \cite{SPAKOWSKI} is essential in the following
we include a proof of:

\begin{lemma}
In a second countable Hausdorff space there exists a
countable basis which Hausdorff separates the compacts.
\end{lemma}

\begin{proof}
The family $\cal B$ of finite unions of sets from a countable
basis is a countable separant for the compacts:
Let $M$ and $N$ be disjoint compact sets.
The Hausdorff property gives disjoint open sets $U'$ and $V'$
with $M \subset U'$ and $N \subset V'$.
$U'$ is a union of sets from the basis.
This gives a covering of $M$ and 
the compactness of $M$ gives a finite subcovering.
The union $U$ of the sets in the finite subcovering belongs to
$\cal B$.
A set $V$  in ${\cal B}$ with the property $N \subset V \subset V'$ is
found likewise.
The conclusion is the existence of two disjoint open sets $U$ and $V$ in
the basis $\cal B$ with $M \subset U$ and $N \subset V$.
\end{proof}

The proof of the following Lemma is a modification of the separation ideas
found in \cite{SPAKOWSKI} where open-measurability of
the intersection of two compact-valued open-measurable multifunctions
into a countably $C$-separated space is proven.
A separable metric space is countably $C$-separated,
but a non-separable metric space is not countably $C$-separated
since a  countably $C$-separated space is separable.
The following holds in particular for multifunctions
into a metric space.

\begin{lemma}
The intersection $F \cap G$ of two closed-valued
compact-measurable multifunctions
into a Hausdorff space with second countable compacts
is a compact-measurable multifunction.
\end{lemma}
\begin{proof} Let $K$ be compact.
The previous Lemma gives a countable family $\cal B$ of open
sets which Hausdorff  separates the compacts in $K$.
If
$(U,V) \in {\cal B}' \Leftrightarrow
 [U,V \in {\cal B}, U \cap V = \emptyset]$,
then
\[
\{\omega \st F (\omega) \cap G (\omega) \cap K = \emptyset \} =
\bigcup_{(U,V) \in {\cal B}'}
\{\omega \st F (\omega) \cap K \cap U^c = \emptyset\; \mbox{and} \;
  G (\omega) \cap K \cap V^c = \emptyset \}.
\]
This proves compact-measurability of the multifunction 
$F \cap G$ since  
$K \cap U^c$ and $K \cap V^c$ are compact sets.
\end{proof}

The next Theorem generalizes from finite intersections of
multifunctions
to countable intersections.
Theorem 3.4 in
\cite{HPV} gives the same result in the case of multifunctions into a
metric space $X$.
Our proof is different and gives a more general result since
a metric space has second countable compacts,
but the converse is false.

\begin{theo}
The intersection of a sequence of closed-valued
compact-measurable multifunctions
into a Hausdorff space with second countable compacts
is a compact-measurable multifunction.
\end{theo}
\begin{proof}
When $K$ is compact,
the set
\[
\{\omega \st [\bigcap_{n \ge 1} F_n (\omega)] \cap K = \emptyset \} = 
\bigcup_{m \ge 1}
\{\omega \st [\bigcap_{1 \le n \le m} F_{n} (\omega)] \cap K = \emptyset \}
\]
is measurable since a finite intersection of closed-valued
compact-measurable multifunctions is compact-measurable.
\end{proof}

Our main result on intersections is
the Souslin intersection Theorem:

\begin{theo}
Let $S$ be a Souslin space.
The intersection 
$\cap_{s \in S} F_s$ of a compact-measurable
family  $\{F_s \}_{s \in S}$ of closed-valued
multifunctions into a Hausdorff space with second countable compacts
is a universally compact-measurable multifunction.
\end{theo}
\begin{proof}
If $K$ is compact and $F_s (\omega)$ is closed, then
\[
\{\omega \st [\cap_{s \in S} F_s (\omega)] \cap K = \emptyset \} = 
\bigcup_{n = 1}^\infty \bigcup_{\vec{s} \in S^n}
\{\omega \st [\cap_{i = 1}^n F_{s_i} (\omega)] \cap K = \emptyset \}.
\]
$S^n$ is a Souslin space,
so the above set is universally measurable from the
Souslin projection theorem and 
the compact-measurability of 
$(\omega,\vec{s}) \mapsto  [\cap_{i = 1}^n F_{s_i} (\omega)]$
which we proceed to prove.
Let
$G_i (\omega, \vec{s}) \dlik F_{s_i} (\omega)$,
so $G_i^{-1} (A) = \{(\omega, \vec{s}) \st 
F_{s_i}(\omega) \cap A \neq \emptyset  \}$.
The multifunction $G_i$ is compact-measurable
from the compact measurability of the family 
$\{F_s \}_{s \in S}$.
The multifunction $\cap_{i = 1}^n G_i$ is then also
compact-measurable.
\end{proof}

The following Theorem is a kind of dual to the previous.

\begin{theo}
A countable intersection of closed-measurable
compact-valued multifunctions into a second countable Hausdorff space
is closed-measurable.
Let $S$ be a Souslin space.
The intersection 
$\cap_{s \in S} F_s$ of a closed-measurable
family  $\{F_s \}_{s \in S}$ of compact-valued
multifunctions into a second countable Hausdorff space
is a universally closed-measurable multifunction.
\end{theo}
\begin{proof}
This is similar to the above.
\end{proof}

The following Theorem generalizes Theorem 3.7 in \cite{SPAKOWSKI}
since a compact-valued multifunction into 
a countably $C$-separated space is open-measurable
if and only if it is closed-measurable.
We remove the Hausdorff assumption,
and go beyond countable intersections.

\begin{theo}
A countable intersection of closed-measurable
compact-valued multifunctions into a countably $C$-separated $T_1$-space
is closed-measurable.
Let $S$ be a Souslin space.
The intersection 
$\cap_{s \in S} F_s$ of a closed-measurable
family  $\{F_s \}_{s \in S}$ of compact-valued
multifunctions into a countably $C$-separated $T_1$-space
is a universally closed-measurable multifunction.
\end{theo}
\begin{proof}
Let $K$ be compact and let $\{U_1, U_2, \ldots\}$ be the 
$C$-separating family of open sets.
For each $x \in K^c$ there is an $i$ with $x \in U_i \subset K^c$,
since $\{x \}$ is closed by the $T_1$ assumption.
This proves that $K$ is closed.
Let $F_i$ be closed-measurable compact-valued multifunctions,
and let $G$ be a closed set.
$F_1 (\omega)$ is closed and $F_2 (\omega) \cap G$
is compact, so
\[
\{\omega \st F_1 (\omega) \cap F_2 (\omega) \cap G = \emptyset \} =
\bigcup_i [ \{\omega \st  F_1 (\omega) \cap U_i^c = \emptyset \}
\cap \{\omega \st  F_2 (\omega) \cap G \cap U_i = \emptyset \}].
\]
This set is measurable, because
$F_2 \cap G$ is closed-measurable and then also open-measurable.
$F_1 (\omega)$ is compact, so
\[
\{\omega \st [\bigcap_{n \ge 1} F_n (\omega)] \cap G = \emptyset \} = 
\bigcup_{m \ge 1}
\{\omega \st [\bigcap_{1 \le n \le m} F_{n} (\omega)] \cap G = \emptyset \},
\]
and the countable intersection result is proven.
The Souslin intersection result follows as before.
\end{proof}

\section{Examples.}

Let $X$ be the set of complex numbers equipped with the
usual topology $\tau$ and family ${\cal B}$ of Borel sets.
Equipp $E \dlik \prod_{i = 1}^\infty X$ with the
$\sigma$-algebra $\otimes_i {\cal B}$ and
the measure $\mu = \otimes_i \mu_i$ where $\mu_i (X) = 1$.

\subsection{Spectrum of a sequence.}

The spectrum is the multifunction
$\Sp : E \into X$ defined by 
\[
\Sp (x) \dlik \{x_1, x_2, \ldots \}^- .
\]
$\Sp$ is closed-valued and open-measurable from
\[
\Sp_i (x) \dlik \{x_i \},
\]
\[
\Sp = [\bigcup_i \Sp_i]^-,
\]
\[
\Sp^{-1} (U^{\mbox{open}}) \stackrel{(1)}{=} [\bigcup_i \Sp_i]^{-1} (U) 
\stackrel{(2)}{=}  \bigcup_i \Sp_i^{-1} (U),
\]
\[
\Sp_i^{-1} (U) = \{x \st x_i \in U \},
\]
since the sets $\{x \st x_i \in U \}$ generate the 
product $\sigma$-algebra.
The equality $\stackrel{(1)}{=}$ proves more generally that
a multifunction is open-measurable if and only if the closure
of the multifunction is open-measurable.
The equality $\stackrel{(2)}{=}$ does not depend on the countability of the
index set and gives more generally the 
measurability of the union of a measurable family of multifunctions.
The corresponding statement for intersections is false and this is
why measurability of the intersection of multifunctions is a more
delicate question.
The intersection relation corresponding to $\stackrel{(2)}{=}$
is
\[
[\bigcap_i F_i]^{-1} (A) \subset \bigcap_i F_i^{-1} (A) .
\]

\subsection{The selection theorem.}

Let $x = (x_1, x_2, \ldots)$ be a sequence of complex valued
measurable functions.
The above gives open-measurability of the multifunction
$\Sigma:\Omega \into X$
\[
\omega \mapsto \Sigma (\omega) \dlik \Sp (x (\omega)) .
\]
Each $x_i$ is a selection for $\Sigma$ since $x_i (\omega) \in \Sigma (\omega)$.
The sequence $x$ is a dense selection for $\Sigma$.

Let $\omega \mapsto F (\omega)$ be a closed non-empty valued open-measurable
multifunction into $X$.
$X = \CompN$ is Polish so the selection theorem gives a dense measurable selection 
$x$ and $F = \Sp (x)$.
The selection is not unique.

\subsection{The essential spectrum.}

The essential spectrum of the sequence $x$ is
\[
\Sp_{ess} (x) \dlik \bigcap_{n \ge 1} \Sp (x_n, x_{n + 1}, \ldots) . 
\]
It follows from the countable intersection theorem 
that $\Sp_{ess}$ is compact-measurable,
since each $x \mapsto \Sp (x_n, x_{n + 1}, \ldots)$ is compact-measurable.
$X = \CompN$ is a countable union of 
compacts,
and $\Sp_{ess}$ is then also open-measurable.

\subsection{The tail part of a multifunction}

Let $\sigma : E \into X$ be a closed-valued open-measurable
multifunction.
The tail dependent part of $\sigma$ is given by
\[
\sigma_{\sim} (x) \dlik \bigcap_{n \ge 1} 
\bigcap_{\vec{y} \in \CompN^{n - 1}} \sigma (\vec{y}, x_n, x_{n + 1}, \ldots ).
\]
$\sigma_{\sim}$ is then universally 
open-measurable from the Souslin intersection theorem.
This conclusion holds also if we only assume
that  $x \mapsto \sigma (x)$ is a closed-valued universally open-measurable
multifunction into $X$:
The set $Y$ of all closed sets in $X$ is equipped with
the $\sigma$-algebra ${\cal F}$ generated by sets
$\{A^{closed} \st A \cap U^{open} \neq \emptyset \}$.
This $\sigma$-algebra is countably generated.
Then there exists a closed-valued open-measurable
$\varsigma$ such that 
$\mu \{x \st \varsigma (x) \neq \sigma (x)  \} = 0$
\cite[p.103]{KECHRIS},
and  
$\mu \{x \st \varsigma_\sim (x) \neq \sigma_\sim (x)  \} = 0$ 
which gives the claim.

The $\sigma$-algebra ${\cal F}$ contains sets 
$Z_1, Z_2, \ldots$ which separates the points in $Y$. 
Kolmogorov's zero-one law \cite[p.104]{KECHRIS}  gives 
$\mu \{x \st \sigma_\sim (x) \in Z_i \} \in \{0, 1 \}$.
We conclude existence of a closed set $A \subset X$
such that 
$\mu \{x \st \sigma_\sim (x) = A \} = 1$.

\section{Multifunctions restricted to equivalence classes.}

Define the relation $\sim$
on $E$ by 
\[
x \sim y \;\;\; \Leftrightarrow \;\;\;
[\exists N\;\; \forall n \ge N\;\; x_n = y_n] .
\]
The relation is an equivalence relation since it is
transitive ($x \sim y, y \sim z \Rightarrow x \sim z $),
reflexive ($x \sim x$), and symmetric ($x \sim y \Rightarrow y \sim x$).
In the previous subsection we established the measurability of
the tail dependent part
\[
\sigma_{\sim} (x) = \bigcap_{y \sim x} \sigma (y)
\]
of $\sigma$.
In the following we will consider a generalization of this.

Let $\sim$ be an equivalence relation on a set $E$ and consider
a multifunction $F: E \into X$.
The $\sim$-part $F_{\sim}$ of $F$ is 
$F_{\sim} (x) = \bigcap_{y \sim x} F (y)$.
It follows that $F_{\sim}$ is
\begin{description}
\item[(i)  Included] $F_{\sim} (x) \subset F (x)$. 
\item[(ii) Invariant] $x \sim y \Rightarrow F_{\sim} (x) = F_{\sim} (y)$. 
\item[(iii) Maximal] If $G$ is included in $F$ and is $\sim$-invariant,
                     then $G \subset F_{\sim}$.
\end{description}
Conversly it follows that a multifunction $H$ which is
included, invariant, and maximal is unique, so
$H = F_{\sim}$.

The measurabilty of $F_{\sim}$ is a fundamental question.
As shown above the Souslin intersection theorem gives a positive result
for compact-measurability in a special case.
A negative result for open-measurabilty follows from the next subsection.

\subsection{A general union result.}

Let $G: \Omega \into E$ and $F: \Omega \times E \into X$ be
multifunctions.
Define the multifunction $H: \Omega \into X$ by
\[
H (\omega) \dlik \bigcup_{b \in G (\omega)} F (\omega, b).
\]
The following identity holds
\[
H^{-1} (U) = \Prj_{\Omega} \left[ 
\Omega \times E \times U \;\;\cap\;\;
\Gr G \times X \;\;\cap\;\;
\Gr F
\right] .
\]
Assume open-measurability of $F$ and $G$
and that $E$ and $X$ are Polish.
It follows that $H$ is universally open-measurable.

Consider next the multifunction
$K: \Omega \into X$ defined by
\[
K (\omega) \dlik \bigcap_{b \in G (\omega)}  F (\omega, b).
\]
A special case is given by $\Omega = E$,
$b \in G(\omega) \Leftrightarrow b \sim \omega$,
and $F (\omega, b) = F (\omega)$,
in which case $K = F_{\sim}$.
Measurability properties of $K$ follows from the identity
\[
\Gr K^c = \Prj_{\Omega \times X} \;\;
 \left[\Gr G \times X\;\; \cap\;\; \Gr F^c \right] .
\]
Assume in addition to the above that $\Omega$ is Polish and
is equipped with the corresponding Borel field.
It follows that $\Gr K$ is Borel if and only if
\[
\Prj_{\Omega \times X} \;\;
\Gr G \times X \;\;\cap\;\; \Gr F^c
\]
is Borel, 
since the complement of a Souslin set is Souslin 
if and only if it is Borel \cite[p.205]{BOURBAKI.III}.
This implies that open-measurability of $K$ fails without
further assumptions,
since there exist Borel sets which have non-Borel projections.
The question of compact-measurability of $K$ is left open.

{\small

}

\hfill \begin{minipage}[t]{0ex}
\begin{tabbing}
Department of Mathematics and Statistics - AVH\\
University of Trondheim\\
N-7055 Dragvoll\\
NORWAY\\
\ \\ \ \\
gunnar@matstat.unit.no
\end{tabbing}
 \end{minipage}

\end{document}